\documentclass[12pt,english]{smfart}
\usepackage{amscd,amsfonts,amsmath,amssymb,amstext,amsthm}
\usepackage{mathpazo}
\usepackage{hyperref}
\usepackage{mathrsfs}

\newtheorem{theorem}{Theorem}[section]
\newtheorem{lemma}[theorem]{Lemma}
\newtheorem{proposition}[theorem]{Proposition}

\theoremstyle{remark}

\renewcommand{\theequation}{\arabic{section}.\arabic{equation}}

\newcommand{\hq } {{/\kern -.185em/}}
\newcommand{\he } {{\kern -.050em\sim _H \kern -.050em }}
\newcommand{\eq } {{\kern -.100em\sim \kern -.100em}}
\newcommand{\eqs } {{\kern -.100em\sim }}
\title[Homogeneous manifolds sastisfying
Levi conditions]{Remarks on homogeneous complex manifolds satisfying 
Levi conditions}
\thanks{The author's research for this article was partially supported by
SFB/TR 12 and Schwerpunkt \emph{Representation Theory} of the DFG}
\author{Alan Huckleberry}
\address{Fakult\"at f\"ur Mathematik, Ruhr-Universit\"at Bochum,\\ 
Universit\"atsstra\ss e 150, 
D-44801 Bochum, Germany, and \newline
School of science and engineering, Jacobs University,
Bremen Compus Ring 1,\newline D-28759 Bremen, Germany}
\email{ahuck@cplx.ruhr-uni-bochum.de}
\begin{document}
\maketitle
\begin {abstract}
Homogeneous complex manifolds satisfying various types of Levi
conditions are considered.  Classical results which were of
particular interest to Andreotti are recalled. Convexity
and concavity properties of flag domains are discussed in some
detail. A precise classification of pseudoconvex flag domains
is given. It is shown that flag domains which are in a certain
sense generic are pseudoconcave.
\end {abstract}

\section {Introduction}\label {introduction}
In the early 1960's Andreotti devoted a great deal of his
attention to complex spaces satifying various types of
Levi-conditions. Major works in this direction include his 
description of the field of meromorphic
functions on a pseudoconcave complex space (\cite {A}), showing that 
it is an algebraic function field of transcendence degree at
most the dimension of the space, and his fundamental work
with Grauert (\cite{AGr1}) on finiteness and vanishing of
cohomology on $q$-convex spaces.  At the time
the case of spaces possessing strictly plurisubharmonic exhaustions,
or, under further assumptions, exhaustions having only a
semipositive Levi-form, were well-understood.
It was indeed quite natural to initiate a study of manifolds which
can be viewed as lying between compact and Stein.

Coming from algebraic geometry Andreotti was interested in 
the examples of such manifolds which are obtained by
removing high codimensional subvarieties from (compact)
projective algebraic manifolds. If $Z$ is compact and $E$ is
the subvariety which is removed, then the set of algebraic cycles
contained in $X:=Z\setminus E$, an open set in the Chow variety
of $X$, is of basic importance.  Transferring cohomology from
the pseudoconcave, $q$-pseudoconvex space $X$ to the level
of functions on its cycle space was the topic of his basic joint
works with Norguet (\cite {AN1,AN2}).  In $\S\ref{pseudoconcave flag domains}$
of this note we underline another setting, that of flag domains,
where cycle spaces, pseudoconcavity and $q$-convexity 
go hand-in-hand.  Our research in this area (see, e.g., \cite {FHW})
strongly relies on Barlet's methods for constructing and dealing with
cycle spaces (see, e.g.,\cite{Ba}), and there is no doubt that
Andreotti's viewpoints on this subject were among the inspiring
factors for Barlet's early works.

Andreotti was well-acquainted with the method of discrete group
quotients for constructing projective or quasi-projective varieties
and was particularly interested in such quotients which arise
in moduli problems. In a jewel which is perhaps only known to
specialists (\cite {AGr2}) he and Grauert introduced the notion
of a \emph{pseudoconcave discrete group action} which is the
appropriate translation to the level of $D$ of the notion that
the discrete group quotient $D/\Gamma $ is pseudoconcave. As
an example they showed that the quotient of the Siegel upper half-plane by
the modular group is pseudoconcave and as a consequence that
interesting spaces of automorphic forms are finite-dimensional.
Borel, who took this result to its appropriate 
level of generality (\cite {Bo}), once enthusiastically recalled to us 
how struck he was with the beauty of this simple idea.

Of course it was clear to Andreotti that the notions of 
pseudoconcavity and/or mixed Levi-conditions are of basic
importance, and that one should employ these as Leitf\"aden for
discovering and analyzing interesting new classes of manifolds.
He also knew that it makes sense to involve symmetry at least in the
initial stages of such considerations.  Thus it should come as no surprise
that at the end of his Stanford course on several complex 
variables, which covered most of the topics mentioned above,
he asked the student who was responsible for the notes 
to look for new classes of pseudoconcave manifolds.  
Typically he suggested an extremely simple starting point: 
complex Lie groups.

The present note is organized as follows.  In $\S\ref{lie groups}$
we discuss the case of Lie groups.  It turns out that one easily sees
that the only such manifolds which are pseudoconcave 
are compact tori. However, this result is not as negative as 
first meets the eye, because in the process of considering candidates
for pseudoconcave Lie groups one meets Levi-degenerate, pseudoconvex 
manifolds
where first examples of interesting, number-theoretic guided
foliations play a role. We continue the discussion of analogous
pseudoconvexity phenomena for nilmanifolds in $\S\ref{nilmanifolds}$
and flag domains in $\S\ref{flag domains}$.  In 
Theorem \ref{pseudoconvex} we show in particular that the 
Remmert reduction of pseudoconvex flag domain is a precisely
defined homogeneous bundle over a Hermitian symmetric space. 

Our work in $\S\ref{pseudoconcave flag domains}$ is devoted to
a study of pseudoconcave flag domains. We suspect that virtually
all flag domains which are not pseudoconvex are in fact
pseudoconcave, but at the present time we are only able to
prove this for flag domains which are in a certain sense
generic ($\S\ref{pseudoconcavity}$). These include 
period domains for moduli problems, e.g., for marked K3-surfaces, 
which were certainly of interest to Andreotti.  
\section {Pseudoconvexity}
As mentioned above, our original starting point was to determine
if there are interesting pseudoconcave complex Lie groups.  In the
first section here we pursue this as a guideline, but in 
fact end up showing that Lie groups are more interesting from the 
point of view of pseudoconvexity. 
In the following paragraph we prove analogous results
for homogeneous nilmanifolds.  The final paragraph is devoted to
a detailed description of pseudoconvex flag domains.  The results
on complex Lie groups are classical, but the proofs given here
underline the importance of Levi-foliations, a theme that is
of recent interest and which also plays a role in our discussion 
of nilmanifolds. Although new, the results on nilmanifolds 
only require implementation of classically known information, 
in particular a basic fact due to Loeb (\cite {L}) concerning 
the relation of geodesic convexity and Levi-pseuodoconvexity in a 
Lie group setting. Our characterization of pseudoconvex flag domains 
utilizes the notion of cycle connectivity which is the
flag domain analog of the condition of 
\emph{rationally connected} in algebraic geometry.

Before going further let us recall that by definition a 
(connected) pseudoconcave complex
manifold $X$ contains a relatively compact open set $Z$ so that
for every point $p\in \mathrm {cl}(Z)$ there is a holomorphic
mapping $\psi :\Delta \to \mathrm {cl}(Z)$ of the unit disk $\Delta $ 
in the complex plane with $\psi (0)=p$ and 
$\psi (\mathrm {bd}(\Delta ))\subset Z$.  The following is the
first basic property of these manifolds.
\begin {proposition}
Pseudoconcave manifolds possess only constant holomorphic
functions.
\end {proposition}
\begin {proof}
Let $f\in \mathcal O(X)$ and note that the restriction of
$\vert f\vert $ to $\mathrm {cl}(Z)$ takes on its maximum
at some point $p$.  Since $p$ is contained in a holomorphic
disk $\psi (\Delta )$ whose boundary lies in $Z$, the
maximum principle implies that $f$ is constant on that
disk and therefore takes on its maximum at a point of the
open set $Z$.  Another application of the maximum principle
implies that $f$ is constant on $X$.
\end {proof}
\subsection {Complex Lie groups}\label{lie groups}
If $G$ is a connected complex Lie group with
$\mathcal O(G)\cong \mathbb C$, then there is no
nonconstant holomorphic homomorphism 
$G\to \mathrm {GL}_\mathbb C(V)$ to the general linear
group of a complex vector space.
This is due to the fact that $\mathrm {GL}_\mathbb C(V)$ is
an open subset of $\mathrm {End}(V)\cong V\otimes V^*$ which
is itself a complex vector space and is in particular
holomorphically separable.  Let $\mathfrak g$ denote the
Lie algebra of $G$ and recall that the adjoint representation
$\mathrm {Ad}:G\to \mathrm {GL}_\mathbb C(\mathfrak g)$
is exactly such a map. Since in general the kernel of 
this map is the center of $G$, we have the following first
remark. 
\begin {proposition} A connected complex Lie
group $G$ with $\mathcal O(G)\cong \mathbb C$ is Abelian.
\end {proposition}
In particular, if $G$ is pseudoconcave, then it is Abelian.

\subsubsection {Background on Abelian Lie groups}
Let us investigate the case of a complex Abelian Lie group
more closely. Recall that
the exponential map $\mathrm {exp}:\mathfrak g\to G$ of
a (connected) Abelian Lie group is a surjective 
homomorphism from the additive group $(\mathfrak g,+)$.
In the case of a complex Lie group the exponential map
is holomorphic and therefore such groups are of the form $V/\Gamma $,
where $\Gamma $ is a discrete additive subgroup of a 
complex vector space $V$.

If $\mathcal O(G)\cong \mathbb C$, then $\Gamma $ must be
rather large.  A key object for understanding its size and
position with respect to the linear complex structure
of $V$ is the real subspace 
$\mathrm {Span}_\mathbb R(\Gamma)=:V_\Gamma $.

Every connected (not necessicarily complex) Lie group $G$ possesses
a compact subgroup $K$ having the property that $G/K$ is 
diffeomorphic to a vector space. One shows that such groups
are maximal compact subgroups and that any two are conjugate.  
If $G=V/\Gamma $ is Abelian, then the torus $V_\Gamma/\Gamma =:=K$
is compact and $G/K$ is diffeomorphic to any (real) subspace of
$V$ which is complementary to $V_\Gamma $. Thus $K$ is the
unique maximal compact subgroup of $G$.

Observe that if $U_2$ is a complex subspace of $V$ which is 
complementary to $U_1:=V_\Gamma +iV_\Gamma$, then $G$ is holomorphically
isomorphic to $U_1/\Gamma \times U_2$.  Thus it is enough to
consider the situation where $\Gamma $ generates $V$ as a complex
vector space, i.e., where $\mathrm {Span}_\mathbb C(\Gamma )=V$.

Define the additive complex subgroup $W_\Gamma $ of $V$ to be
the maximal complex subgroup in $V_\Gamma $, in other words
$W_\Gamma=V_\Gamma \cap iV_\Gamma $, and regard the $W_\Gamma $-orbits
in $G$ as foliating the torus $K$. As abstract complex manifolds
the leaves of this foliation are all equivalent to the orbit
of the identity which is the subgroup $W_\Gamma /(W_\Gamma \cap \Gamma)$.
The closure of this orbit is a subtorus $K_1$, and we may split
$K$ as a product $K=K_1\times K_2$ of two subtori where $K_2$
is a totally real subgroup of $G$. Let $\mathfrak k_j$ be
the uniquely defined subspaces of $V$ so that 
$K_j=\mathfrak k_j/\Gamma $ and define 
$V_j:=\mathfrak k_j+i\mathfrak k_j$, $j=1,2$.  
Assuming that $\Gamma $ generates $V$ as a complex vector
space, it therefore follows that 
$G=V/\Gamma =V_1/(V_1\cap \Gamma)\times V_2/(V_2\cap \Gamma )$.

Let us summarize the above discussion in the proof of the following
decomposition theorem of Remmert-Morimoto (see \cite {K,M}).
\begin {theorem}\label{decomposition}
A connected Abelian complex Lie group $G=V/\Gamma $ is
the direct product $G=G_1\times G_2\times G_3$, where
$G_3\cong (\mathbb C^n,+)$, $G_2\cong ((\mathbb C^*)^m, \cdot )$
and $\mathcal O(G_1)\cong \mathbb C$.
\end {theorem}
\begin {proof}
The complex group $G_3$ arises (noncanonically) as the
complement of the canonically defined complex subspace
$U_1=V_\Gamma +iV_\Gamma$.  The factor $U_1/\Gamma $ is canonically
embedded in $G$ and it contains the canonically defined
complex subgroup $G_1=V_1/(V_1\cap \Gamma)$.  The noncanonical
splitting $K=K_1\times K_2$ of the maximal compact subgroup
defines the complementary complex subgroup 
$G_2=V_2/(V_2\cap \Gamma )$.  Since $K_2$ is totally real,
$G_2\cong (\mathbb C^*)^m$.

Finally, recall that $K_1$ 
is foliated by the dense orbits of the complex subgroup
$W_\Gamma $.  If $f\in \mathcal O(G_1)$, then we consider
its restriction to $K_1$ which is compact so that 
the restriction of $\vert f\vert $ takes on its maximum at
some point $p\in K_1$.  Since the orbit map $W_\Gamma \to G$,
$w\to w(p)$, is holomorphic, the pullback of $f$ to $W_\Gamma$
is holomorphic and thus the maximum principle implies
that this pull-back is constant.  Consequently, $f$
is constant on the (dense!) $W_\Gamma $-orbit of $p$ and
is therefore constant on the torus $K_1$.  But 
$K_1=\mathfrak k_1/(\mathfrak k_1\cap \Gamma )$ and
$V_1=\mathfrak k_1+i\mathfrak k_i$. Hence, it follows from
the identity principle that the pullback
of $f$ to $V_1$ is constant and consequently $f$ is
constant on $V_1/(V_1\cap \Gamma)=:G_1$. 
\end {proof}
\subsubsection {Cousin groups}
Restricting to the Abelian case and regarding 
$\mathbb C^n$ and $(\mathbb C^*)^m$ as being well-understood,
in the notation of the above decomposition theorem 
it is reasonable to further restrict to the case where $G=G_1$. 
Of course the case where $G$ is
compact has a long history, but it was first in the early
20th century that Cousin called attention to interesting
complex analytic phenomena in the noncompact case 
(see \cite {C}).  Thus if $\mathcal O(G)\cong \mathbb C$, we refer
to $G$ as being a \emph{Cousin group}.

Let us now turn to the matter of pseudoconcavity/pseudoconvexity 
of Cousin groups.  Recalling the notation above,  in this
situation $G=V/\Gamma $ and $V=V_\Gamma +iV_\Gamma $ where 
$V_\Gamma =\mathrm{Span}_\mathbb R(\Gamma )$. The maximal complex
subgroup $W_\Gamma=V_\Gamma\cap iV_\Gamma $ has dense orbits
in the maximal compact subgroup $K=V_\Gamma /\Gamma $. Let
$C$ be a complementary subspace of $V_\Gamma $ such that 
$V_\Gamma =W_\Gamma\oplus C$.  Note that $C$ is totally real and
that $V$ decomposes as a complex vector space as
$$
V=W_\Gamma \oplus (C\oplus iC)\,.
$$
Let $\eta $ be an exhaustion of $iC$ which is defined as the
norm-squared function of some (positive-definite) inner-product.
If we regard $\eta $ as being defined on $V$, then its full
Levi-form $i\partial \bar \partial \eta $ is positive-semidefinite
with degeneracy $W_\Gamma $.

Now define $\rho :G\to \mathbb R^{\ge 0}$ by pulling back $\eta $
by the linear projection to $iC$.  Its properties are summarized
as follows.
\begin {proposition}
A noncompact Cousin group $G=V/\Gamma $ possess a
plurisubharmonic exhaustion $\rho :G\to \mathbb R^{\ge 0}$
which is invariant by the maximal compact subgroup
$K=V_\Gamma /\Gamma $.  For every $p\in G$ the degeneracy of
the Levi-form of $\rho $ at $p$ is the tangent space to
the orbit $W_\Gamma .p$ of the maximal complex subgroup
$W_\Gamma=V_\Gamma \cap iV_\Gamma $ which is dense in $K.p$.
\end {proposition}
In particular, noncompact Cousin groups are pseudoconvex in
a very strong sense and the following is therefore immediate.
\begin {proposition}\label{pseudoconcave lie groups}
Pseudoconcave complex Lie groups are compact complex tori.
\end {proposition}
\begin {proof}
Let $G$ be a pseudoconcave Lie group and assume that it 
is noncompact with a relatively compact open subset $Z$
defining its pseudoconcavity.  For $r>0$ define
$B_r:=\{\rho <r\}$, where $\rho $ is the exhaustion function
defined above, and let $r_0:=\mathrm {inf}\{r:B_r\supset Z\}$.
Thus $\mathrm {cl}(Z)\subset \mathrm {cl}(B_{r_0})$ and
there exists $p\in \mathrm {bd}(Z)$ which is also contained
in the level surface $M_{r_0}:=\{\rho=r_0\}$.  Let 
$\psi :\Delta \to \mathrm {cl}(Z)$ be the holomorphic disk
at $p$ which is guaranteed by the pseudoconcavity and note that
$\widehat \rho :=\psi ^*(\rho )$ is plurisubharmonic on $\Delta $ with 
$\widehat \rho (0)=r_0$. Since $\widehat \rho <r_0$ on all
boundary points of $\Delta $, this violates the maximum 
principle.
\end {proof}
Before concluding this paragraph we should make a number
of remarks on the history of this subject, particularly focused
on Andreotti's involvement. First, the proof of 
Proposition \ref{pseudoconcave lie groups} in (\cite {AH})
differs somewhat from the one above: If $G=V/\Gamma $ is a noncompact
Cousin group, then (in the words of Andreotti) \emph{we can introduce
a small earthquake} and move $\Gamma $ to a nearby group
$\Gamma _\varepsilon $ so that the resulting variety 
$G_\varepsilon =V/\Gamma _\varepsilon$ is still pseudoconcave,
but $\mathcal O(G)\not\cong \mathbb C$. It should be
noted that it can be arranged that such an earthquake
produces a holomorphically convex manifold so that
the Levi-problem has a positive answer for a dense set
of discrete groups.

Andreotti was interested in the fields of
meromorphic functions of Cousin groups, in particular
in relation to their projective algebraic 
equivariant compactifications (quasi-abelian varieties)
and the role of $\Theta $-functions (see \cite{AGh}).
It is should be mentioned that $\Theta $-theory is
in general not adequate for describing the
meromorphic functions on a (noncompact) Cousin group.
On the other hand, since the time of Cousin there have
been a number of interesting developments (see \cite{AK}).

Due to our focusing on topics of particular interest to
Andreotti, we have covered only a very small part of 
the interesting early results involving Lie theoretic
considerations in complex analysis.  In closing this paragraph
we would, however, like to note one further result which
underlines the fact that the decomposition
of Theorem \ref{decomposition} can be viewed in a much more
general context (see \cite {MM} and \cite {M}).
\begin {theorem}\label{holomorphic reduction}
Let $G$ be a connected complex Lie group equipped with
the holomorphic equivalence relation $x\eq y$ if and only
if $f(x)=f(y)$ for all $f\in \mathcal O(G)$.  Then the
quotient $G\to G/\eq $ is given as a holomorphic group 
fibration $G\to G/C$ where the fiber $C$ is a closed, central
Cousin subgroup of $G$ and the base is a Stein Lie group.
\end {theorem}
\subsection {Nilmanifolds}\label{nilmanifolds}
Here we carry through virtually the same line of discussion
for nilmanifolds as that above for Abelian groups. 
By definition a \emph{complex nilmanifold} $X$ is 
homogeneous under the holomorphic action of
a connected complex nilpotent Lie group, i.e., $X=G/H$,
where $G$ is a connected complex nilpotent group and
$H$ is a closed complex subgroup.  We may assume that
the $G$-action on $X$ is almost effective in the sense
that there are no positive-dimensional normal subgroups
of $G$ which are contained in $H$.  In other words, the
subgroup of elements in $G$ which fix every point of $X$
is at most discrete.

One of the first steps toward understanding any complex
homogeneous space $G/H$ is to consider the normalizer
fibration $G/H\to G/N$ where $N$ is the normalizer
in $G$ of the connected component $H^0$.  If we consider
the action of $G$ on $\mathfrak g$ by the adjoint representation
and regard $\mathfrak h$ as a point in the Grassmannian
$\mathrm{Gr}_k(\mathfrak g)$ where $k=\mathrm{dim}_\mathbb C\mathfrak h$,
then the base $G/N$ is the orbit of that point.  In particular,
the base of the normalizer fibration is an orbit via a $G$-representation
in the projective space $\mathbb P(\wedge ^k\mathfrak g)$.

A connected solvable Lie group $G$ acting via a
linear representation of a vector space $V$ stabilizes
a full flag $0\subset V_1\subset \ldots \subset V_{m-1}\subset V$.
of subspaces (Lie's Flag Theorem). Thus, for example, 
a $G$-orbit in the associated projective space 
$\mathbb P(V)$ either lies in the \emph{affine}
space $\mathbb P(V)\setminus \mathbb P(V_{m-1})$ or
is contained in the smaller projective space $\mathbb P(V_{m-1})$.
Repeating this until reaching the point where the
orbit in question is not contained in some $\mathbb P(V_\ell)$
of the flag, we have the following first remark.
\begin {proposition}
Orbits of a connected complex solvable group acting as a group
of holomorphic transformations on a projective space
$\mathbb P(V)$ are holomorphically separable.  In particular,
if $X=G/H$ is a homogeneous manifold under the holomorphic
action of a complex solvable Lie group and 
$\mathcal O(X)\cong \mathbb C$, then, assuming that the
$G$-action is almost effective, it follows that $H$ is
discrete.
\end {proposition}
Thus as in the case of an Abelian group, if we are guided by
investigating the possibility of a nilmanifold $X$ being
pseudoconcave, we may assume that it is of the form
$G/\Gamma $ where $\Gamma $ is discrete.  So let us now
restrict our considerations to such manifolds.

If $X=G/\Gamma $ is a nilmanifold with discrete isotropy and,
without loss of generality, $G$ is
simply-connected, then
$\mathrm {exp}:\mathfrak g\to G$ is a biholomorphic (in fact
algebraic) map.  In analogy to the Abelian case,
realizing $\Gamma $ as a discrete subset
of $\mathfrak g$, it spans a (real) Lie subalgebra 
$\mathfrak g_\Gamma $ such that the associated group $G_\Gamma $
contains $\Gamma $ with $G_\Gamma /\Gamma $ 
compact (Theorem of Malcev-Matsushima).

Continuing with the analogy to the Abelian case we consider
the complex Lie algebra 
$\widehat {\mathfrak g}_\Gamma:=\mathfrak g_\Gamma +i\mathfrak g_\Gamma $
and the associated complex subgroup $\widehat G_\Gamma $.
As a quotient of simply-connected complex nilpotent groups
$G/\widehat G_\Gamma $ is biholomorphically equivalent to some $\mathbb C^n$.
The bundle $G/\Gamma \to G/\widehat G_\Gamma $ is holomorphically trivial
and therefore there is no loss of generality to assume that
$G=\widehat G_\Gamma$, i.e., that the Lie algebra level $\Gamma $ generates
$\mathfrak g$.

The key subalgebra for complex analytic considerations is
$\mathfrak m=\mathfrak g_\Gamma \cap i\mathfrak g_\Gamma $.
It is an \emph{ideal} in $\mathfrak g$!  Of course, just as in
the Abelian case, the action of the associated group $M$
on $G/\Gamma $ can be wild. However, if we replace 
$G$ by $N:=G/M$ and $G_\Gamma $ by $N_\mathbb R:=G_\Gamma/M$,
then $N_\mathbb R$ is a real form of $N$.  In this situation
in the Abelian case we identified the analog of $N/N_\mathbb R$ with $iC$ 
and lifted to $N$ an exhaustion which is defined on $iC$
by a positive-definite inner product. In that case 
straightforward computations show that the lifted exhaustion
has the expected plurisubharmonicity. In the nilpotent
case at hand we must apply Loeb's Theorem (\cite {L}) which
states that since the adjoint representation of $N_\mathbb R$
has purely imaginary spectrum (the eigenvalues are all zero!),
it follows that there is a smooth exhaustion $\eta $ 
of $N/N_\mathbb R$ which lifts to a strictly plursubharmonic 
function on $N$. 

Let us now review the situation discussed above 
where $X=G/\Gamma $, $G=\widehat G_\Gamma $
and $M$ is the normal closed complex subgroup of $G$ defined by
the ideal $\mathfrak m$. Here we have the fibration
$$
X= G/\Gamma \to G/G_\Gamma =N/N_\mathbb R
$$
and we pull back the exhaustion guaranteed by Loeb's theorem
to an exhaustion $\rho $ of $X$ which we view as
a $G_\Gamma $-invariant function on $G$.  Since this is defined
on $G$ by lifting a strictly plurisubharmonic function from
$N=G/M$, it is plurisubharmonic on $G$ and therefore $\rho $
is a plurisubharmonic exhaustion of $X$. Hence we have
the following result.
\begin {proposition} If $G=\widehat G_\Gamma $, then a smooth exhaustion
$\rho :X\to \mathbb R^{\ge 0}$ guaranteed by Loeb's theorem
is plurisubharmonic with Levi-degeneracy determined by
the ideal $\mathfrak m$. In particular, $X=G/\Gamma $ is 
pseudoconvex.
\end {proposition}
Several remarks are now in order.  First, we arrived
at the the situation where $G=G_\Gamma $ by splitting
off a factor of $\mathbb C^n$ from an arbitrary nilmanifold
of the type $G/\Gamma $. In fact one doesn't need the
assumption of discrete isotropy for such a splitting,
i.e., every complex nilmanifold is a product of $\mathbb C^n $
and a nilmanifold of the form $G/\Gamma $ with $G=\widehat G_\Gamma $
(\cite {LOR}).  Thus we have the following

\smallskip\noindent
{\bf Zusatz.} Every complex nilmanifold possesses a 
smooth plurisubharmonic exhaustion.

It should also be underlined that due to the nonabelian nature
of the situation, the Levi-foliations defined by $\rho $ should
be much more interesting that those in the Abelian case.

Now recall that we originally were investigating the
possibilty of complex homogeneous spaces being pseudoconcave
but ended up with a general pseudoconvexity result.  Thus
with exactly the same proof as that for 
Proposition \ref{pseudoconcave lie groups} we have the following
remark.
\begin {proposition} Pseudoconcave nilmanifolds are compact.
\end {proposition} 
Finally, e.g. in the case of discrete isotropy, $X$ is Stein
if and only if $G_\Gamma $ is totally real (\cite {GH}).
Furthermore, in analogy to Theorem \ref{holomorphic reduction},
as in the case of Lie groups a general complex homogeneous manifold
$G/H$ has a canonically defined holomorphic 
reduction $X=G/H\to G/I=X/\eq $. In the nilpotent case the
fiber possesses only the constant holomorphic functions
and the base is Stein (\cite {GH}).  This is far from being
true in the general situation.
\subsection {Flag domains}\label {flag domains}
\subsubsection {Background}\label{flag domain background}
Recall that the radical $R$ of a connected Lie group
$G$ is defined to be the maximal connected solvable normal
subgroup of $G$.  If $R$ is trivial, i.e., consists only
of the identity, then $G$ is said to be \emph{semisimple}.
A fundamental difference between solvable and semisimple
groups is that most semisimple groups possess intrinsic 
algebraic structure whereas solvable groups do not.
In general a Lie group $G$ is a product $R\cdot S$ of its radical
and a maximal semisimple subgroup $S$.  In fact, $S$
is unique up to conjugation.  The intersection $R\cap S$ 
is a discrete central subgroup of $G$ and if, for example, $G$
is simply-connected, then this project is a semidirect
product $G=R\ltimes S$.

Above we commented on certain aspects of the solvable case,
i.e., where the complex Lie group $G$ agrees with its radical.
If $G$ is semisimple, $H$ is a complex closed subgroup and
$X=G/H$, then the assumption of existence of meromorphic or 
plurisubharmonic functions on $X$ or even that $X$ is K\"ahler 
is very restrictive. In most cases this forces $H$ to be an
algebraic subgroup of $G$ (\cite {Be,BeO}).  For example it
is known that $X$ is Stein if and only if $H$ is reductive.
In the other extreme of Levi conditions,
even under the further condition that $H$ is algebraic
there is no known characterization of $X=G/H$ being
pseudoconcave.

The situation changes dramatically if $G$ is allowed to
be a real semisimple group. In that setting the first basic
examples arise as \emph{flag domains}.  Here we describe 
the flag domains which possess plurisubharmonic exhaustions
and in the following section we discuss flag domains
with Levi conditions in the opposite direction, e.g.,
pseudoconcavity.  Let us begin with a sketch of some 
background information. The first basic results on
flag domains can be found in (\cite {W}). A systematic
treatment, which in particular gives the details of the
results needed here, is presented in (\cite {FHW}).

Let us begin with a real Lie group $G_0$
and consider an action $G_0\times X\to X$ by holomorphic
transformations on a complex manifold.  If this action
is transitive, then we refer to $X$ as being $G_0$-homogeneous.
In that case we may as usual identify $X$ with $G_0/H_0$
where $H_0$ is the isotropy group at a base point. However,
unlike the case where $X=G/H$ is the homogeneous space
under the holomorphic action of a complex Lie group,
the complex structure of $X$ is not transparently encoded
in the Lie group structure.

At the level of vector fields the situation is slightly better,
because the complexified Lie algebra, 
$\mathfrak g:=\mathfrak g_0+i\mathfrak g_0$, is represented
as an algebra of holomorphic $(1,0)$-vector fields on $X$.
In other words, the complexified Lie group $G$ acts locally
and holomorphically on $X$. To put this in perspective 
consider the example of the standard $G_0=\mathrm{SU}(1,1)$-action
on $\mathbb P_1(\mathbb C)$ and let $X$ be one of its two
open orbits (both are disks!).  Here, as in the general case, 
the complexification $G=\mathrm{SL}_2(\mathbb C)$ acts locally
on $X$ and in addition has the advantage of acting globally
on $\mathbb P_1$.  One regards the holomorphic $G$-manifold
$\mathbb P_1$ as the \emph{globalization} of the local
$G$-manifold $X$.

There is a beautiful theory of globalization of local
actions due to Palais which was adapted to our complex analytic
setting by Heinzner and Iannuzzi (see \cite {HI}).  However,
even when $X$ is $G_0$-homogeneous it is difficult to know
whether or not it is embedded in a $G$-globalization.  On the
other hand, as reflected by the example of the unit disk in 
$\mathbb P_1$, the case where a globalization is implicitly
given is already quite interesting.  The case of flag domains
is one such situation.

In order to discuss flag domains we restrict to the case
where $G_0$ is semisimple.  Due to standard splitting theorems
it is usually enough to assume, as we do here, that it is
even simple. For our purposes it is also enough to consider
the situation where it is embedded in its complexification
$G$.  We let $G\times Z\to Z$ be a holomorphic $G$-action
on a complex manifold and consider the induced $G_0$-action.
A case of fundamental interest, e.g., for studying
the representation theory of $G_0$, is that where $Z$ is
assumed to be a compact, $G$-homogeneous projective manifold.
Choosing a base point we write $Z=G/Q$.

Whereas much is known about \emph{flag manifolds}
$Z=G/Q$ of the above form, restricting to the $G_0$-action 
adds significant complications which lead to new phenomena
which are not yet understood.  The following first key step
is, however, proved by classical combinatorial arguments
(see \cite{W}).
\begin {proposition} The real form $G_0$ has only finitely
many orbits on the flag manifold $Z$.
\end {proposition}
In particular, $G_0$ has open orbits in $Z$.  We refer to
such as a \emph{flag domain} and, if there is no confusion,
will always denote it by $D$.  One purpose of this paper
is to give evidence for the following (perhaps naive) 
conjecture.

\emph{Flag domains are either pseudoconvex or pseudoconcave.}
 
By \emph{pseudoconvex} we mean that there exists an
exhaustion $\rho :D\to \mathbb R^{\ge 0}$ which is 
plurisubharmonic outside of a compact set. \emph{Pseudoconcavity}
is understood in the usual sense of Andreotti 
(see $\S\ref{introduction}$).  Below we give a detailed 
description of the pseudoconvex flag domains.
After doing so, we devote the remainder of the paper 
to describing a large class of pseudoconcave flag domains and 
to giving some indication of the validity of the conjecture.
\subsubsection {Background on cycle spaces}
Our discussion of pseudoconvex flag domains $D$
makes strong use of the \emph{cycles} in $D$ 
which are defined by the actions of $G_0$ and $G$.
Here we begin by introducing minimal background on this
subject, referring the reader to (\cite {FHW}) for detailed
proofs.

Let $K_0$ be a maximal compact subgroup of $G_0$. Any two such 
are $G_0$-conjugate and as a result for our purposes the choice
is not relevant. A basic fact, which is just the tip of the
iceberg of Matsuki duality, is that there is a unique $K_0$-orbit
in $D$ which is a complex submanifold.  Let us refer to it
as the \emph{base cycle} $C_0$, regarded as either a submanifold
or a point in the cycle space of $D$. In the sense of dimension
$C_0$ is \emph{the} minimal $K_0$-orbit in $D$.  If $K$ denotes
the complexification of $K_0$ which is realized as a subgroup
of $G$, then $C_0$ is also a $K$-orbit. It can be characterized
as the only $K$-orbit of a point in $D$ which is contained in
$D$.

Here not much information is needed about
the cycle spaces at hand. However, let us introduce some
convenient notation which will also be of use in the next
section. For this let $q:=\mathrm{dim}_\mathbb C C_0$ and
let $\mathcal C_q(D)$ be the space of $q$-dimensional cycles
in $D$. Recall that such a cycle is a linear combination
$C=n_1X_1+\ldots +n_kX_k$ where the $X_j$ are irreducible
$q$-dimensional subvarieties and the coefficients $n_j$
are positive integers. In a natural way
$\mathcal C_q(D)$ is a complex space which can be
regarded as an open subset of the cycle space 
$\mathcal C_q(Z)$.  Our view of these cycle spaces
is that of (\cite {Ba}).  The reader is also referred to
Chapter 8 of \cite{FHW} for a minimal presentation.

In our particular case $\mathcal C_q(Z)$
is smooth at $C_0$ (see Part IV of \cite{FHW}) and thus
it makes sense to speak of \emph{the} irreducible component
of $\mathcal C_q(D)$ at $C_0$.  We simplify the notation
by replacing $\mathcal C_q(D)$  by this irreducible component.  
Since the algebraic group
$G$ is acting algebraically on $Z$, it acts algebraically
on the associated cycle spaces $\mathcal C_q(Z)$. 
The group-theoretical cycle space
$\mathcal M_D$ which, for example, is of basic interest in 
representation theory is defined as the connected component
of the intersection of the orbit $G.C_0$ with $\mathcal C_q(D)$.
It is in fact a closed submanifold of 
$\mathcal C_q(D)$ (\cite {HoH}).

\subsubsection {Cycle connectivity}
We say that two points $x,y\in D$ are \emph{connected by cycles} 
if there are cycles $C_1,\ldots ,C_m\in \mathcal M_D$ so that
the union $C_1\cup \ldots \cup C_m$ is connected with $x\in C_1$
and $y\in C_m$.  The relation defined by $x\eq y$ if and only
if $x$ and $y$ are connected by cycles is an equivalence
relation. Since $\mathcal M_D$ is $G_0$-invariant, it
is by definition $G_0$-invariant.  Thus, if we choose
a base point $z_0$ in $D$ and identify $D$ with
$G_0/H_0$ where $H_0$ is the $G_0$-isotropy at $z_0$,
then the quotient $D\to D/\eq $ defined by the equivalence
relation is given by a homogeneous fibration 
$G_0/H_0\to G_0/I_0$ where $I_0$ is the stabilizer of
the equivalence class $[z_0]$.  

By definition the equivalence classs $[z_0]=I_0/H_0=:F$ is
a closed real submanifold of $D$. Note that if $\Omega $
is a relatively compact open neighborhood of $z_0$ in $F$,
then there is an open neighborhood $U$ of the identity
of the isotropy group $G_{z_0}$ which maps $U$ into $F$.
Since $G_{z_0}$ has only finitely many orbits in $Z$
and since $G_{z_0}$ is complex, this imples that
$F$ contains an open dense subset which is a complex
submanifold of $D$.  But $I_0$ acts transitively and holomorphically
on $F$ and therefore $F$ is a complex submanifold of $D$.
The stabilizer in $\mathfrak g$ of $F$, i.e., the stabilizer
of $F$ under the local $G$-action, is a complex Lie subalgebra
$\widehat {\mathfrak q}$ which contains the algebra 
$\mathfrak q$ of the $G$-isotropy subgroup at $z_0$.  
Consequently, there exists a globally defined complex
subgroup $\widehat Q$ so that the fiber $F$ at the base point
of $D\to D/\eq $ is an open $I_0$-orbit in the (compact) 
fiber of $G/Q\to G/\widehat Q$ at the base point.
\begin {proposition}
The cycle connectivity reduction $D\to D/\eq=\widehat D$ of a flag
domain is given by the restriction of a canonically defined
$G$-equivariant map $Z=G/Q\to G/\widehat Q=\widehat Z$. 
It is a holomorphic map onto a 
$G_0$-flag domain $\widehat D$ in $\widehat Z$.  
In particular, the fibers of $D\to D/\eq$
are themselves connected complex manifolds.
\end {proposition}
\begin {proof}
Except for one point the proof is given above: We must
show that the intersection of the fibers of the
$G$-equivariant map $Z\to \widehat Z$ with $D$ are
connected.  But this follows immediately from the
fact that $\widehat D$ is simply-connected
(see \cite{W} or \cite{FHW}).
\end {proof}
Since the base cycle $C_0$ is a $K$-orbit and in particular
$k(z_0)\eq z_0$, we know that $K$ stabilizes $[z_0]$. In other
words, $K\subset \widehat Q$ and it follows that the
base cycle $\widehat C_0$ in $\widehat D$ is just a single point.
Since it is known that this can only happen when 
$\widehat D$ is a $G_0$-Hermitian symmetric space of noncompact
type embedded in its compact dual $\widehat Z$ (\cite {W}).
Let us note this for future reference.
\begin {proposition}
Either $D=G_0/H_0$ is cycle connected, i.e., any two points are
connected by a chain of cycles in $\mathcal M_D$ or
the cycle connectivity equivalence reduction $D\to \widehat D$
is such that $\widehat D=G_0/K_0$ is a Hermitian symmetric
space embedded in its compact dual $\widehat Z$ and the
neutral fiber $K_0/H_0=C_0$ is the base cycle itself.
\end {proposition}
\begin {proof}
We know that $K_0$ fixes the base point in $\widehat D$
and by general theory the isotropy group of a $G_0$-symmetric
space is exactly a maximal compact subgroup of $G_0$.
\end {proof}
As a consequence we see that if $D$ is not cycle connected,
then any two cycles either agree or are disjoint.
In other words, in that case the fibers of the reduction 
$D\to D/\eq$ are cycles and the cycle space $\mathcal M_D$ 
is the Hermitian symmetric space $\widehat D$.
\subsubsection {Pseudoconvex flag domains}
\label {pseudoconvex flag domains}   
Let us say that a complex manifold $X$ is \emph{pseudoconvex}
if it possesses a continuous proper exhaustion function 
$\rho :X\to \mathbb R^{\ge 0}$ which is plurisubharmonic
on the complement $X\setminus S$ of a compact set $S$.
It should be underlined that, even if $\rho $ is
smooth, we are only assuming the semi-positivity of its
Levi-form.

Given the preparation in the previous paragraph, it is
now a simple matter to characterize pseudoconvex flag
domains.  For this the following is the main remark.
\begin {lemma}
Cycle connected flag domains possess only constant
plurisubharmonic functions.
\end {lemma}
\begin {proof}
Let $D$ be a pseudoconvex flag domain and consider a
plurisubharmonic function $\rho $ on $D$.  Given 
two points $x,y\in D$, connect them with a chain
$C_1,\ldots , C_m$ of cycles. Since $\rho \vert C_i$
is constant for every $i$, it is immediate that
$\rho (x)=\rho (y)$.
\end {proof}
\begin {proposition}
Cycle connected flag domains are not pseudoconvex.
\end {proposition}
\begin {proof}
Given a cycle connected flag domain $D$, assume to
the contrary that it is pseudoconvex. Let 
$\rho :D\to \mathbb R^{\ge 0}$ be an exhaustion 
which is plurisubharmonic on $D\setminus S$.
Define $r_0=\mathrm {min}(\rho \vert S)$ and define
$\widehat \rho$ to be the maximum of $\rho $ and
the constant function $r_0 +1$. Then, contrary to
the above Lemma, $\widehat \rho$
is a nonconstant plurisubharmonic function on $D$.
\end {proof}
It follows that pseudoconvex flag domains have 
cycle reduction $\pi :D\to \widehat D$ to a Hermitian
symmetric space $\widehat D$. The unique cycle through
a given point $z\in D$ is the $\pi $-fiber $\pi ^{-1}(\pi (p))$
through that point.  Since $\widehat D$ is a contractible
Stein manifold, the bundle $\pi :D\to \widehat D$ is trivial
and $D$ can be (noncanonically) realized as the product
$C_0\times \widehat D$.  In summary we have the following
characterization of pseudoconvex flag domains.
\begin {theorem}\label {pseudoconvex}
For a flag domain $D$ the following are equivalent.
\begin {enumerate}
\item
$D$ is pseudoconvex
\item
$D$ is holomorphically convex with Remmert reduction
$D\to \widehat D$ to a Hermitian symmetric space.
\item
$D$ is not cycle connected with cycle reduction agreeing
with its Remmert reduction.
\item
$D$ possesses a nonconstant plurisubharmonic function.
\end {enumerate}
\end {theorem}
It should be underlined that domains fulfilling any one of
the above conditions are of the form $D=G_0/L_0$ where
$L_0$ is a \emph{compact} subgroup of the group $G_0$ which is
of Hermitian type.  As a result one can also describe such
domains via root-theoretic data (see \cite {W,FHW}).
\section {Pseudoconcave flag domains}
\label {pseudoconcave flag domains}
Above we began our study of flag domains from the point
of view of Levi-geometry by showing that pseudoconvex
flag domains are of a very special nature (
Theorem \ref{pseudoconvex flag domains}). As a Leitfaden
for further investigations we conjecture that if a flag
domain is not pseudoconvex, then it is pseudoconcave.
Here we begin with a brief exposition of constructions
of two natural exhaustions of flag domains whose Levi-curvature
is at least in principle computable. Then, using the 
exhaustion constructed using cycle geometry,
we describe a rather large class of flag domains which
are pseudoconcave.  We underline that further information
concerning properties of these exhaustions in a general setting
would certainly be of interest.   
\subsection {Exhaustions}
Here we discuss two natural methods for constructing
$K_0$-invariant exhaustions of flag domains.  From the
point of view of Levi-geomtry both
have their advantages and disadvantages.  The first 
was introduced by Schmid for a flag domain $D$ which
is a $G_0$-orbit in $Z=G/B$ where $B$ is a Borel subgroup
of $G$ (\cite {S}).  This was generalized to 
\emph{measureable} flag domains in (\cite {SW}, see also
$\S4.6$ in \cite {FHW}).  This type of exhaustion has the
advantage that it is smooth and clearly $q$-convex in
the sense of (\cite {AGr1}).  However, determining
the concavity properties requires root calculations which
vary from case to case and which could be rather subtle.

Exhaustions of a second type were recently constructed in
(\cite {HW}). These are canonically related to a given irreducible
$G$-representation and the Levi-geometry of $\mathcal M_D$.
They have the disadvantage of only being continuous, but
they are $q$-convex in a very strong sense and, as shown
in the sequel, their concavity properties (which are
related to cycle geometry) are more transparent than
those of the exhaustions of the first type.  
\subsubsection {Schmid-Wolf exhaustions}
As above $D$ denotes a flag domain which is a $G_0$-orbit
in a flag manifold $Z=G/Q$.  The first observation relevant
for the construction of the Schmid-Wolf exhaustion is the
fact that the anticanonical bundle $K^{-1}\to Z$ is very ample.
Assuming that we have chosen $G$ to be simply-connected,
this is a $G$-bundle $G\times _\chi \mathcal C$ where
$\chi : Q\to \mathbb C^*$ is an explicitly computable
character.  Recall that if $h$ is a Hermitian bundle metric
(unitary structure) on a line bundle $L\to X$ on a complex manifold
with associated norm-squared function $\Vert \cdot \Vert ^2$, 
then the Chern form $c_1^h(L)$ is the negative of the Levi-form
$\frac{i}{2}\partial \bar \partial \mathrm {log}(\Vert \cdot \Vert ^2)$ of the
exhaustion $\mathrm {log}(\Vert \cdot \Vert ^2)$ of the 
bundle space.

In the case at hand, having fixed a Cartan involution
$\theta $ on $\mathfrak g_0$ which defines the Lie algebra
$\mathfrak k_0$ of the maximal compact subgroup $K_0$, we
extend $\theta $ to a holomorphic involution of $\mathfrak g$
and define the antiholomorphic involution $\sigma :=\tau \theta $,
where $\tau $ is the antiholomorphic involution which
defines the real form $\mathfrak g_0$ on $\mathfrak g$.
The Lie group $G_u$ corresponding to 
$\mathfrak g_u:=\mathrm{Fix}(\sigma )$ is the maximal 
compact subgroup of $G$ which is canonically associated
to the real form $G_0$ with the choice of maximal compact
subgroup $K_0$.

If $L\to Z$ is any nontrivial $G$-bundle on $Z=G/Q$, then
$G$ has exactly two orbits in the bundle space $L$, the
0-section, which corresponds to the fixed point of $Q$
in the neutral fiber, and its complement. In this complement
all $G_u$-orbits are real hypersurfaces and $G_u$ acts
transitively on the 0-section as well.  Define $V_u:=G_u\cap Q$
so that $Z=G/Q=G_u/V_u$. Writing 
$L$ as a $G_u$-bundle, $L=G_u\times _\chi \mathbb C$, we
note that, since the restriction of $\chi $ to $V_u$
is nontrivial and there is 
a unique $S^1$-invariant unitary structure on $\mathbb C$
normalized at $1\in \mathbb C$,
there is an essentially unique $G_u$-invariant unitary
structure on $L$.  The level surfaces of the associated
norm-squared function are exactly the $G_u$-orbits 
in $L$.  If, as in the case $L=K^{-1}$, the bundle $L$
is ample, then Chern form is positive-definite or, equivalently,
from the point of view of the 0-section 
the $G_u$-hypersurface orbits are strongly pseudoconcave.

Let us now consider the restriction of the 
anticanonical bundle $K^{-1}$ to a flag domain $D$.
It is a $G_0$-homogeneous bundle $G_0\times _{\widehat \chi}\mathbb C$.
Here $\widehat \chi $ is a $\mathbb C^*$-valued character
from the $G_0$-isotropy $V_0=G_0\cap Q $. One is of course interested
in the situation where $\widehat \chi $ is $S^1$-valued so
that, as in the case of the $G_u$-bundle on $Z$, the anticanonical 
bundle on $D$ would possess a $G_0$-invariant unitary structure.
The condition for this is called \emph{measurable} (\cite {W},
see also $\S4.5$ in \cite {FHW}).

There are a number of equivalent conditions for $D$ to be
measurable (\cite {W}). Here are those of a less technical
nature:
\begin {enumerate}
\item
$D$ possesses a $G_0$-invariant pseudok\"ahlerian
metric.
\item
$D$ posseses a $G_0$-invariant volume form.
\item
The isotropy group $V_0$ is reductive in the sense that
its complexification $V$ is a complex reductive subgroup
of $G$.
\item
The isotropy group $V_0$ is the centralizer of a compact
subtorus $T_0\subset G_u\cap V_0$ so that
$D$ is realized as a coadjoint orbit.
The symplectic form induced from this realization is
the invariant form defined by the pseudok\"ahlerian metric.
\end {enumerate}
One can show that if one flag domain in $Z$ is measurable,
then all others flag domains in $Z$ are also measurable.  
Thus measurable is a property of the $G_0$-action on the flag manfold $Z$.
For example, flag manifolds $Z=G/B$ are measurable for
any real form.  Furthermore, every flag manifold $Z$ is measurable
if $G_0$ is of Hermitian type. On the other hand it is
seldom the case that a flag manifold $Z$ is measurable for
$G_0=\mathrm {SL}_n(\mathbb R)$. 
 
Now if $Z$ is measurable and $D$ is a flag domain in $Z$,
then we have two Hermitian norm-squared functions on
its anticanonical bundle, the restriction $\Vert \cdot \Vert ^2_u$
of the $G_u$-invariant norm on the full anticanonical bundle of
$Z$ and the $G_0$-invariant function $\Vert \cdot \Vert^2_0$
coming from the coadjoint symplectic form or from the
the pseudok\"ahlerian metric.  The characters which define these
norms are actually defined on the same torus $T_0$ which splits
off of both isotropy groups and on that torus they are the same.
Thus the ratio 
$$
R=\frac{\Vert \cdot \Vert ^2_0}{\Vert \cdot \Vert ^2_u}
$$
is a well-defined function on the base $D$ and one can
show that $\rho :=\mathrm {log}(R)$ is an exhaustion function
of $D$.  We refer to this as the Schmid-Wolf exhaustion
of $D$ (see \cite{S,SW}). Note that since $h_u$ is $G_u$-invariant
and $h_0$ is $G_0$-invariant, $\rho $ is invariant with respect
to the maximal compact subgroup $K_0=G_u\cap C_0$.

The Levi-form of $\rho $ is the difference 
$c_1^{h_u}-c_1^{h_0}$. A direct calculation with roots
shows that $c_1^{h_0}$ is of signature $(q,n-q)$ where 
$q$ is the dimension of the cycle $C_0$ and $n=\mathrm {dim}(D)$.
Since $c_1^{h_u}>0$, the exhaustion $\rho $ is $q$-complete in the sense of 
Andreotti and Grauert, i.e., at every point of $D$
the Levi-form of $\rho $ has at least $n-q$ positive. 
Let us note this result.
\begin {theorem}
The Schmid-Wolf exhaustion of a measurable flag domain $D$ 
is $q$-complete in the sense of Andreotti and Grauert.
\end {theorem}
The Schmid-Wolf exhaustions have the advantage that one can
directly apply the Andreotti-Grauert vanishing theorem for 
higher cohomology groups.  One disadvantage is that without
further root-theoretic computation one does not know the
degree of concavity.  Furthermore, one only knows the
existence of these exhaustions on measurable domains.  
 
\subsubsection {Exhaustions via Schubert slices}
\label{Schubert exhaustions}
Here we explain the construction of (\cite {HW}) which
uses cycle space geometry to produce an exhaustion
$\rho _D:D\to \mathbb R^{\ge 0}$ of any given flag domain.
It has the disadvantage of only being continuous, but
it is $q$-pseodoconvex in a strong sense. Its concavity
properties are closely related to the cycle geometry 
of $D$.

The construction of $\rho _D$ requires basic information
concerning \emph{Schubert slices}.  We sketch this here
and refer the reader to $\S9$ of (\cite {FHW}) for details.
In order to define a Schubert slice we must recall that
$G_0$ possesses an Iwasawa-decomposition $G_0=K_0A_0N_0$.
Here $K_0$ is a maximal compact subgroup as above, $A_0$
is an Abelian subgroup noncompact type, $N_0$ is a certain
nilpoint subgroup defined by root-theory and which is normalized
by $A_0$.  Writing $K$, $A$ and $N$ for the complexifications
of these subgroups which are subgroups of $G$, we note the
fundamental fact that the set $KAN$ is a \emph{proper} 
Zariski open subset of $G$.

Now if $C_0:=K_0.z_0$ is a base cycle, then every orbit
of $A_0N_0$ in $D$ must have nonempty intersection with
$C_0$. The following is basic for our discussion 
(see $\S7.3$ in \cite {FHW} for details). 
\begin {theorem}
The set $I$ of points $z\in C_0$ which are such that
the orbit $A_0N_0.z$ is of minimal dimension under all
$A_0N_0$-orbits in $D$ is finite. For every $z\in I$
the orbit $\Sigma :=A_0N_0.z$ has the following properties:
\begin {enumerate}
\item
$\Sigma $ is closed in $D$ and open in $AN.z$ which
is a Schubert cell in $Z$.
\item
The intersection of $\Sigma $ with every cycle $C\in \mathcal M_D$
consists of exactly one point and $C$ is transversal to $\Sigma $
at that point.
\end {enumerate}
\end {theorem} 
For obvious reasons we refer to the orbits $\Sigma $ as
Schubert slices.  We should note that the Schubert cells
in the above statement are meant to be the orbits of
Borel groups $B$ which contain an Iwasawa-factor $AN$.
These are very special Borel groups, being those whose
fixed point is in the (unique) closed $G_0$-orbit in $Z$.

If $r _\Sigma :\Sigma \to \mathbb R^{\ge 0}$ is a strictly
plurisubharmonic function on $\Sigma $, then we define
a plurisubharmonic function 
$\rho _\Sigma :\mathcal M_D\to \mathbb R^{\ge 0}$ by
$\rho _\Sigma (C):=r_\Sigma (\sigma _C)$, where $\sigma _C$
is the unique point of intersection of $C$ and $\Sigma $.
After checking that $\rho _\Sigma $ is a plurisubharmonic function
on $\mathcal M_D$ one might hope that if $\Sigma $ is Stein
and $r_\Sigma $ is an exhaustion, then $\rho _\Sigma $ might be
a plurisubharmonic exhaustion of $\mathcal M_D$.  Simple
examples, e.g., the one interesting flag domain defined
by the $\mathrm {SU}(2,1)$-action on the 3-dimsional manifold
of full flags in $\mathbb C^3$, show that in general $\Sigma $ 
is not Stein. Furthermore, even if $r_\Sigma $ is an exhaustion, $\rho _\Sigma $
may not be an exhaustion.

The difficulties mentioned above can be remedied by simultaneously 
considering a number of Schubert slices.  To do this we start
with strictly plurisubharmonic functions $r_\Sigma $ which arise
in a natural way, in this case associated to an irreducible 
representation of $G$.  For this we recall that if $L\to Z$
is a $G$-line bundle, then the $G$-representation on
$\Gamma (Z,L)$ is irreducible.  Conversely, every irreducible
holomorphic representation of $G$ occurs in this way.

Now recall that a given $\Sigma =A_0N_0.z$ is open in 
the Schubert cell $\mathcal O_S:=B.z\cong \mathbb C^{n-q}$
which closes up to the Schubert variety $S$. Given an
of an ample bundle $L\to Z$, we let $V$ be the
space of sections of $L\vert S$ which are defined
as restrictions of sections of $L$ on $Z$.
Let $s\in V$ be a $B$-eigenvector which is not
identically zero. It follows that $s$ vanishes
exactly on $S\setminus \mathcal O_S$ (see, e.g., \cite{FHW}, $\S7.4C$)
and if we equip $L$ with the canonically defined $G_u$-invariant
norm-squared function $\Vert \cdot \Vert ^2$, then 
the restriction of 
$
r_\Sigma :=s^*(\mathrm{log}(\Vert \cdot \Vert ^2)$ is a 
strictly plurisubharmonic exhaustion of the Schubert cell 
$\mathcal O_S$. The associated function $\rho _\Sigma $
on the cycle space is plurisubharmonic, but normally not
an exhaustion.  Thus we define
$\rho _{\mathcal M_D}$ to be the supremum of the $\rho _\Sigma $
as $\Sigma $ ranges over \emph{all} possible Schubert slices
for a fixed Iwasawa component $A_0N_0$ and over all Iwasawa
decompositions. Since this is a compact family of Schubert slices,
it can be shown that $\rho _{\mathcal M_D}$ is a continuous
plurisubharmonic function.  Using our analysis of the
boundary behavior of the Schubert slices 
(see, e.g., $\S9.2$ in \cite{FHW}), one proves the following
first result.
\begin {proposition}
The function $\rho _{\mathcal M_D}=\mathrm{sup}_\Sigma (\rho_\Sigma)$
associated to an irreducible representation
of $G$ on the space of sections
of an ample bundle on $Z$ is a continuous plurisubharmonic
$K_0$-invariant exhaustion of the cycle space $\mathcal M_D$.
\end {proposition}
The procedure for transferring $\rho _{\mathcal M_D}$ back to
the domain $D$ is quite natural. For this we let
$\mathfrak X:=\{(z,C)\in D\times \mathcal M_D: z\in C\}$
and denote by $\mu :\mathfrak X\to D$ and 
$\nu :\mathfrak X\to \mathcal M_D$ the canonical projections.
Note that the fiber $\mu ^{-1}(p)=F_p$ can be identified with
the set of cycles in $D$ which contain the point $p$.  Now
define $\rho _{\mathfrak X}:=\rho _{\mathcal M_D}\circ \nu $
and let 
$$
\rho _D(p):=\mathrm {inf}_{F_p}(\rho _{\mathfrak X})\,.
$$
The following can be proved by tracing through the construction
of $\rho _D$ (\cite {HW}).
\begin {proposition}
The function $\rho _D:D\to \mathbb R^{\ge 0}$ is a continuous
$K_0$-invariant exhaustion of $D$ which is $q$-pseudoconvex
in the following sense: For every $r<0$ and every $z$ in the
boundary of the sublevel set $\{\rho _D<r\}$ there exists a
neighborhood $U=U(z)$ and a smooth function $h$ on $U$ such
that $h(z)=r$, $h\le \rho _D\vert U$ and the Levi-form
$L(h)$ restricted to the complex tangent space of $\{h=r\}$
at $z$ has an $(n-q)$-dimensional positive eigenspace.
\end {proposition}
It would be useful if either $\rho _D$ could be smoothed to
an exhaustion which is $q$-pseudoconvex in the sense of
Andreotti-Grauert or if the finiteness/vanishing theorems
of Andreotti-Grauert could be proved under the assumption
of a continuous exhaustion with the above pseudoconvexity 
property. 
\subsubsection {Flag domains are q-pseudoflat}
In (\cite {HN}) we introduced the notion of q-pseudoflatness
as a weakening of both q-Leviflatness and q-pseudoconcavity
(See \cite{HSt} for elementary complex analytic properties 
of such manifolds.).  By definition a q-pseudoflat (connected) 
complex manifold $X$ is required to contain a relatively compact 
open set $Z$ such that every point $p$ of its closure 
$\mathrm {cl}(Z)$ is contained in a $q$-dimensional 
(locally defined) analytic set $A_p$ which itself is
contained in $\mathrm {cl}(Z)$. Examples
on the pseudoconvex side which possess exhaustions by
plurisubharmonic functions whose level sets are foliated by
$q$-dimensional leaves are given by the Lie groups in 
(\ref{lie groups}), the nilmanifolds in (\ref{nilmanifolds})
and the flag domains in (\ref{pseudoconvex flag domains}).  
In the flag domain case the number $q$ is the dimension
of the base cycle $C_0$. An optimal dichotomy might be
that a flag domain is either $q$-Leviflat as in 
Theorem \ref{pseudoconvex} or $q$-pseudoconcave.  We
have stated a weakened version of this conjecture in
(\ref{flag domain background}) and prove the pseudoconcavity
(without any particular degree $q$) of certain flag domains 
in (\ref{pseudoconcavity}). Here we note the following general
result on $q$-pseudoflatness. Its proof follows
by direct inspection of the definition of an exhaustion
defined by the Schubert-slice method.
\begin {proposition}
Let $\rho _D:D\to \mathbb R^{\ge 0}$ be an exhaustion 
of a flag domain $D$ which is defined
by the Schubert-slice method and $D_r=\{\rho_D<r\}$ be
a sublevel set.  Then every $p\in \mathrm{bd}(D_r)$ is 
contained in a cycle $C\in \mathcal M_D$ which itself is
contained in $\mathrm {cl}(D_r)$.  In particular, 
$D$ is $q$-pseudoflat.
\end {proposition}
\begin {proof}
If $\rho _D(p)=r$, then by definition there exists a 
cycle $C_p\in \mathcal M_D$ with $p\in C_p$ such that 
$$
\rho_\mathfrak X(p,C)=r=
\mathrm {min}\{\rho _\mathfrak X(p,C): C\in F_p\}\,.
$$ 
Now consider another point $\widehat p\in C$ and note
that, since $C\in F_{\widehat p}$ and
$$
\rho _\mathfrak X(p,C)=\rho _\mathfrak X(\widehat p,C)
=\rho _{\mathcal M_D}(C)\, ,
$$
it follows that $\rho _D(\widehat p)\le \rho _D(p)=r$, i.e.,
$C\subset \mathrm{cl}(D_r)$.
\end {proof}
\subsection {Pseudoconcavity via cycles} \label{pseudoconcavity}
Here we prove that $D$ is pseudoconcave if it is cycle
connected in a certain strong sense which we refer to
as \emph{generically 1-connected}.  To define this notion first
note that for $p$ an arbitrary point in $D$ 
and $C$ an arbitrary cycle in $G.C_0$ the set of cycles 
in $G.C_0$ which contain $p$ is just the orbit 
$G_{p}.C$ of the $G$-isotropy group at $C$.  We therefore
say that $D$ is generically 1-connected if $C$ has
nonempty intersection with the open $G_{p}$-orbit in $Z$.
One checks that this notion does not depend on the choice
of $p$ or $C$.

Throughout this paragraph we assume that $D$ is 
generically 1-connected. Under this assumption we will 
show that $D$ is pseudoconcave in the sense of Andreotti, i.e.,
that $D$ contains a relatively compact open subset 
$\mathrm{int}(\mathcal K)$ such that for every point
of its closure $\mathcal K$ there is a 1-dimensional
holomorphic disk $\Delta $ with $p$ at its center such
that $\mathrm {bd}(\Delta)$ is contained in 
$\mathrm{int}(\mathcal K)$.   In fact the construction
is such that 
every $p\in \mathcal K$ is contained in a cycle $C$ in
$\mathcal M_D$ which is itself contained in $\mathcal K$.
This cycle has the further property that 
$C\cap \mathrm{int}(\mathcal K)\not=\emptyset$.  Hence,
in a certain sense one may regard $D$ as
being q-pseudoconcave.  At the present time, however,
we are not able to replace $C$ with a q-dimensional polydisk.

The compact set $\mathcal K$ is constructed as follows.
For $p_0$ an arbitrary point in $C_0$ let $U$ be a relatively
compact open neighborhood of the identity in the 
isotropy subgroup $G_{p_0}$. Choose $U$ to be sufficiently small
so that  
$$
\mathcal K_{p_0}:=\{u(p); p\in C_0,\ u\in \mathrm{cl}(U)\}
$$
is contained in $D$ and let
$$
\mathcal K:=\cup_{k\in K_0}k.K_{p_0}\,.
$$

\begin {proposition}
The $K_0$-invariant set $\mathcal K$ is a compact subset
of $D$ which is the closure of its interior 
$\mathrm{int}(\mathcal K)$. The base cycle $C_0$ is
contained in $\mathrm {int}(\mathcal K)$ and every point 
of $\mathcal K$ is contained in a cycle $C$ which is 
contained in $\mathcal K$ and which has nonempty 
intersection with $C_0$.
\end {proposition}
\begin {proof}
Since $\mathcal K=\{ku(p);k\in K_0,u\in \mathrm {cl}(U),p\in C_0\}$
and $K_0$, $\mathrm{cl}(U)$ and $C_0$ are compact, it is
immediate that $\mathcal K$ is compact.  If $z=ku(p)\in \mathcal K$,
then we let $\{p_n\}$ be a sequence in $C_0$ which is contained
in in the open orbit of $G_p$ and which converges to $p$.  
It follows that $z_n:=ku(p_n)$
is in the interior of $\mathcal K$ and $z_n\to z$.  Thus
$\mathcal K$ is the closure of its interior 
$\mathrm{int}(\mathcal K)$.  By definition every point
of the intersection of $C_0$ with the open $G_{p_0}$-orbit
is in $\mathrm{int}(\mathcal K)$. Thus, since $K_0$ acts
transitively on $C_0$, it follows that 
$C_0\subset \mathrm{int}(\mathcal K)$.  Finally, every
point $z\in \mathcal K$ is of the form $z=kuk^{-1}k(p_1)$,
where $p_1\in C_0$. Thus $z\in kuk^{-1}(C_0):=C\subset \mathcal K$. 
Since $kuk^{-1}$ fixes $k(p_0)$, it follows that 
$C\cap C_0\not=\emptyset$.
\end {proof}
In order to replace the \emph{supporting cycles}
with q-dimensional polydisks, the construction of 
$\mathcal K$ may have to be refined.  However, without
further refinements we are able to construct 1-dimensional
supporting disks at each boundary point of $\mathcal K$.
\begin {theorem}
Generically 1-connected flag domains are pseudoconcave.
\end {theorem}
\begin {proof}
Let $z\in \mathrm {bd}(\mathcal K)$, choose 
$C=gC_0g^{-1}\in \mathcal M_D$
to be contained in $\mathcal K$ with $z\in C$ and 
$C\cap C_0\not=\emptyset$, and let $z_0$ be in this 
intersection. Choose a 1-parameter unipotent subgroup
of $gKg^{-1}$ whose orbit of $z_1$ has nonempty intersection
with the open orbit of the isotropy subgroup of $gKg^{-1}$
at $z_1$ and define $Y$ to be the closure of this orbit.
After an injective normalization, $Y$ is just a copy
of $\mathbb P_1$.  Replacing $Y$ by $h(Y)$ where $h$
is in the $gKg^{-1}$-isotropy group at $z_1$, i.e., 
by conjugating the 1-parameter subgroup by $h$, we may 
assume that $Y$ contains points $y_0$ which are arbitrarily 
near $z_0$.  Choose $y_0$ in $\mathrm{int}(\mathcal K)$ and define
$\Delta $ to be the complement in $Y$ of the closure of a 
disk about $y_0$ which is likewise in $\mathrm{int}(\mathcal K)$.
\end {proof}
Our feeling is that most flag domains \emph{are} 
1-connected and that the domains which are not
1-connected can be classified by elementary root
computations.  The argument in the proof of the
following remark gives some indication of this.
\begin {proposition}
Every flag domain of $\mathrm {SL}_n(\mathbb R)$ is
1-connected.
\end {proposition}
\begin {proof}
Let $z_0\in C_0$ be the base point and assume that it
cannot be connected to some point $z\in D$ by a
cycle.  It follows that $C_0$ has empty intersection
with the open $Q$-orbit. In particular, it has empty
intersection with the open orbit of every Borel subgroup
$B$ contained in $Q$.  Thus $C_0$ is contained in
an irreducible $B$-invariant complex hypersurface $H$ in $Z$.  
Given such a hypersurface there is a maximal parabolic
subgroup $\widehat Q$ containing $Q$ such that $H$ is
the preimage of the unique $B$-invariant hypersurface $\widehat H$
in $\widehat Z:=G/\widehat Q$ by the projection $G/Q\to G/\widehat Q$.

We may assume that the unipotent radical $U:=R_u(\widehat Q)$
is contained in $B$ and note that it is Abelian and acts
freely and transitively on the open $B$-orbit in $\widehat Z$.
Thus it acts with 1-dimensional ineffectivity on $\widehat H$.
By construction the base cycle $\widehat C_0$ is contained
in $\widehat H$.  Thus the stabilizer of $\widehat C_0$
in $G$ acts on $\widehat C_0$ with nontrivial ineffectivity.
On the other hand $K=\mathrm {SO}_n(\mathbb C)$ is a 
simple Lie group and consequently this stabilizer contains
$K$ as a proper subgroup.

As a result the domain $\widehat D=G_0.\widehat z_0$ is
of Hermitian holomorphic type (see $\S5$ in \cite {FHW})
and in particular $G_0$ is of Hermitian type, a contradiction.
\end {proof}
The reader will note that, except for the fact that we
use the simplicity of $K$, the discussion in the above 
proof is completely general. However, even in the case
where $K$ is not simple, we would expect that it still 
would be possible to reduce to the Hermitian holomorphic
case.  A concrete example of this is the action of 
$G_0=\mathrm {SO}(3,19)$ on the 20-dimensional quadric
$Z=G/Q$ which is also 1-connected.  The flag domain $D$ 
of positive lines in $Z$ is the moduli space of marked 
K3-surfaces, an example of interest to Andreotti. It also
should be mentioned that if stronger conditions are imposed,
then fine classification results can be proved.  For example,
for the case of strong pseudocavity see (\cite {HS1,HS2}).

\end {document}